\documentclass[12pt]{article}

\usepackage{amsmath,amssymb,amsthm,amssymb}
\usepackage{tikz,tkz-graph,tikz-cd,tikz-qtree}
\usepackage{youngtab}
\usepackage[OT2,T1]{fontenc}

\newtheorem{theorem}{Theorem}[section]
\newtheorem{lemma}{Lemma}[section]
\newtheorem{corollary}{Corollary}[section]

\newtheorem{definition}{Definition}[section]
\newtheorem{proposition}{Proposition}[section]
\newtheorem{remark}{Remark}[section]

\begin{document}
	
\bibliographystyle{abbrv}
	
\title{The Minimal Binomial Multiples of Polynomials over Finite Fields}
	
\author{Li Zhu$^{1}$ and Hongfeng Wu$^{2}$\footnote{Corresponding author.}
\setcounter{footnote}{-1}
\footnote{E-Mail addresses:
lizhumath@pku.edu.cn (L. Zhu), whfmath@gmail.com (H. Wu)}
\\
{1.~School of Mathematical Sciences, Guizhou Normal University, Guiyang, China}
\\
{2.~College of Science, North China University of technology, Beijing, China}}
	
\date{}
\maketitle
	
\thispagestyle{plain}
\setcounter{page}{1}	
	
\begin{abstract}
	Let $f(X)$ be a nonconstant polynomial over $\mathbb{F}_{q}$, with a nonzero constant term. The order of $f(X)$ is a classical notion in the theory of polynomials over finite fields, and recently the definition of freeness of binomials of $f(X)$ was given in \cite{Mart\'{i}nez}. Generalizing these two notions, we introduce the definition of the minimal binomial multiple of $f(X)$ in this paper, which is the monic binomial with the lowest degree among the binomials over $\mathbb{F}_{q}$ divided by $f(X)$. Based on the equivalent characterization of binomials via the defining sets of their radicals, we prove that a series of properties of the classical order can be naturally generalized to this case. In particular, the minimal binomial multiple of $f(X)$ is presented explicitly in terms of the defining set of the radical of $f(X)$. And a criterion for $f(X)$ being free of binomials is given. As an application, for any positive integer $N$ and nonzero element $\lambda$ in $\mathbb{F}_{q}$, the $\lambda$-constacyclic codes of length $N$ with minimal distance $2$ are determined.
\end{abstract}

\section{Introduction}
Let $\mathbb{F}_{q}$ be a finite field with $q$ elements, and $f(X)$ be a nonconstant polynomial with a nonzero constant term over $\mathbb{F}_{q}$. The order of $f(X)$ is defined to be the smallest positive integer $m$ such that $f(X) \mid X^{m}-1$. It is a classical and important notion in the theory of polynomials over finite fields.

Recently, motivated by the study of primitive $k$-normal elements in finite fields, Mart\'{i}nez, Reis and Ribas \cite{Mart\'{i}nez} introduced the definition of freeness of binomials. The polynomial $f(X)$ is said to be free of binomials if it does not divide any binomial with degree less than $m = \mathrm{ord}(f)$. The problem to explore the connection between the polynomials free of binomials and constacyclic codes is proposed in \cite{Mart\'{i}nez}.

A first answer to this problem may be given as follow. Let $f(X)$ be a divisor of $X^{m}-\lambda$, where $\lambda$ is a nonzero element in $\mathbb{F}_{q}$. Then $f(X)$ gives rise to a $\lambda$-constacylic code $C = (f(X))$ of length $m$. This code has minimal distance $2$ if and only if $f(X)$ divides a binomial with degree less than $m$. In particular, if $\lambda=1$, this condition amounts to that either the order of $f(X)$ is less than $m$ or $f(X)$ is not free of binomials.

To depict the constacyclic codes of minimal distance $2$ in general, it is natural to consider the binomial with the lowest degree among the binomials divided by $f(X)$. Hence we propose the following definitions. Clearly, this definition is also a direct generalization of the order of $f(X)$.

\begin{definition}
	Let $f(X)$ be a nonconstant polynomial over $\mathbb{F}_{q}$ with a nonzero constant term. Define the minimal binomial multiple of $f(X)$ to be the monic binomial $X^{n}-\lambda \in \mathbb{F}_{q}[X]$ which is of the lowest degree among the binomials over $\mathbb{F}_{q}$ divided by $f(X)$. The degree $n$ is called the minimal binomial order of $f(X)$, and is denoted by $\mathrm{ord}^{\mathrm{b}}(f) = n$.
\end{definition}

In the present manuscript, we first generalize the correspondence between equal-difference cyclotomic cosets and irreducible binomials over $\mathbb{F}_{q}$, which is given in \cite{Zhu 2}, to the general case as follow.

\begin{theorem}
	Let $f(X)$ be a nonconstant polynomial over $\mathbb{F}_{q}$ that is coprime to $X$, and let $m$ be the order of the radical $\mathrm{rad}(f)$ of $f(X)$. Then $f(X)$ is a binomial if and only if there is a nonnegative integer $v$ such that $f(X) = \mathrm{rad}(f)^{p^{v}}$ and the defining set $T_{\mathrm{rad}(f)}$ of $\mathrm{rad}(f)$ is an equal-difference subset of $\mathbb{Z}/m\mathbb{Z}$.
\end{theorem}
 
Based on this equivalent characterization of binomials, we study the basic properties of the minimal binomial multiple of a nonconstant polynomial $f(X)$ with a nonzero constant term. In particular, we present the minimal binomial order and the minimal binomial multiple of $f(X)$ explicitly, in terms of the defining set of the radical of $f(X)$. 

\begin{theorem}
	Let $f(X)$ be a nonconstant polynomial over $\mathbb{F}_{q}$, which is coprime to $X$, with the irreducible factorization over $\mathbb{F}_{q}$ given by
	$$f(X) = f_{1}(X)^{h_{1}}\cdots f_{t}(X)^{h_{t}}.$$
	Let $m$ be the order of $\mathrm{rad}(f)$, and $c_{m/q}(\gamma_{i})$ be the $q$-cyclotomic coset modulo $m$ defining the irreducible factor $f_{i}(X)$ for $1 \leq i \leq t$. Set $u$ to be the smallest integer such that $p^{u} \geq h_{i}$ for all $i = 1,\cdots,t$, $d_{i} = \gamma_{i+1}-\gamma_{i}$ for $i=1,\cdots,t-1$, and
	$$d_{f} = \mathrm{gcd}(d_{1},\cdots,d_{t-1},m,q-1).$$
	Then the minimal binomial order of $f(X)$ is $p^{u}\cdot\frac{m}{d_{f}}$, and the minimal binomial multiple of $f(X)$ is given by
	$$X^{\frac{p^{u}m}{d_{f}}}-\zeta_{m}^{\frac{p^{u}\gamma m}{d_{f}}},$$
	where $\gamma$ can be choose to be any $\gamma_{i}$, $1 \leq i \leq t$.
\end{theorem}

As an application, in Section \ref{sec 1} we determine the $\lambda$-constacylic codes of length $N$ with minimal distance $2$, where $N$ is a positive integer and $\lambda$ is a nonzero element in $\mathbb{F}_{q}$.

\section{Basic definitions and notations}
Let $n$ be a positive integer. We denote by $\mathbb{Z}/n\mathbb{Z}$ the ring of residue classes modulo $n$. As there is a canonical surjection
$$\mathbb{Z} \rightarrow \mathbb{Z}/n\mathbb{Z}; \ \gamma \mapsto \overline{\gamma} = \gamma+n\mathbb{Z},$$
if making no confusion, we sometimes identify $\gamma$ with its image $\overline{\gamma}$. Therefore the notation $\gamma \in \mathbb{Z}/n\mathbb{Z}$ actually stands for the image of $\gamma$ in $\mathbb{Z}/n\mathbb{Z}$.

If $m$ and $n$ are coprime integers, we denote by $\mathrm{ord}_{n}(m)$ the order of $m$ in the multiplicative group $(\mathbb{Z}/n\mathbb{Z})^{\ast}$, i.e., the smallest positive integer such that 
$$m^{\mathrm{ord}_{n}(m)} \equiv 1 \pmod{n}.$$

For any prime number $\ell$, denote by $v_{\ell}(n)$ the $\ell$-adic valuation of $n$, i.e., the maximal integer such that $\ell^{v_{\ell}(n)} \mid n$. 

For integers $a < b$, we write $[a,b]$ for the integer interval $\{a,a+1,\cdots,b\}$.

Let $p$ be a prime number, and $q=p^{e}$ be a power of $p$, where $e \in \mathbb{N}^{+}$. Let $\mathbb{F}_{q}$ be a finite field of $q$ elements. The multiplicative group $\mathbb{F}_{q}^{\ast}$ of nonzero elements in $\mathbb{F}_{q}$ is a cyclic group. For any $\lambda \in \mathbb{F}_{q}^{\ast}$, the order $\mathrm{ord}(\lambda)$ of $\lambda$ is defined to be the order of the subgroup $\langle\lambda\rangle$ generated by $\lambda$. In other words, $\mathrm{ord}(\lambda)$ is the smallest positive integer such that $\lambda^{\mathrm{ord}(\lambda)}=1$.

Given a positive integer $n$ not divisible by $p$, there are $n$ roots of $X^{n}-1$ lying in some finite extension of $\mathbb{F}_{q}$. A root $\zeta_{n}$ of $X^{n}-1$ which fails to be a root of $X^{m}-1$ for any $m < n$ is called a primitive $n$-th root of unity. Throughout this paper we fix a family 
$$\{\zeta_{n} \ | \ \mathrm{gcd}(n,q)=1\}$$
of primitive roots of unity, which is compatible in the sense that for any $m,n \in \mathbb{N}^{+}$ such that $m \mid n$ and $\mathrm{gcd}(n,q)=1$ it holds that 
$$\zeta_{n}^{\frac{n}{m}} = \zeta_{m}.$$

Let $f(X)$ be a nonconstant polynomial over $\mathbb{F}_{q}$, which is coprime to $X$.  Then there is a smallest positive integer $n$ such that $f(X) \mid X^{n}-1$. The integer $n$ is called the order of $f(X)$ and is denoted by $n=\mathrm{ord}(f)$. The polynomial $f(X)$ has no repeated root if and only if $n$ is coprime to $q$. If, furthermore, $f(X)$ is irreducible over $\mathbb{F}_{q}$, then $\mathrm{ord}(f) = \mathrm{ord}(\delta)$ for any root $\delta$ of $f(X)$ lying in the finite extension field $\mathbb{F}_{q}[X]/(f(X))$ of $\mathbb{F}_{q}$.

Assume that $f(X)$ is a nonconstant squarefree polynomial coprime to $X$ over $\mathbb{F}_{q}$, with $\mathrm{ord}(f)=n$. Then $f(X)$ can be written as
$$f(X) = a\cdot \prod_{\gamma \in T_{f}}(X - \zeta_{n}^{\gamma})$$
for some $a \in \mathbb{F}_{q}^{\ast}$ and $T_{f} \subseteq \mathbb{Z}/n\mathbb{Z}$. That is, up to a multiplication of scalar, $f(X)$ is fully determined by the set $T_{f}$. We refer $T_{f}$ as to the primitive defining set of $f(X)$. Here we call $T_{f}$ primitive since for any $m < n$ there is no subset $T \subseteq \mathbb{Z}/m\mathbb{Z}$ with which $f(X)$ can be written as $f(X) = a \cdot \prod\limits_{\gamma \in T}(X-\zeta_{m}^{\gamma})$.

Given a positive integer $N$ such that $n \mid N$ and $\mathrm{gcd}(q,N)=1$, there is a canonical embedding
$$i_{n,N}: \mathbb{Z}/n\mathbb{Z} \rightarrow \mathbb{Z}/N\mathbb{Z}; a+n\mathbb{Z} \mapsto a\cdot \frac{N}{n} +N\mathbb{Z}.$$
Then $f(X)$ can be also presented as
$$f(X) = a\cdot\prod_{\gamma \in T_{f}}(X - \zeta_{n}^{\gamma}) = a\cdot\prod_{\widetilde{\gamma} \in i_{n,N}(T_{f})}(X-\zeta_{N}^{\widetilde{\gamma}}),$$
where $i_{n,N}(T_{f})$ denotes the image of $T_{f}$ under $i_{n,N}$ in $\mathbb{Z}/N\mathbb{Z}$. In such sense, we call $i_{n,N}(T_{f})$ the defining set modulo $N$ of $f(X)$. If $N$ is clear from the context, we also simply refer $i_{n,N}(T_{f})$ as to a defining set of $f(X)$.

\begin{lemma}\label{lem 2}
	Let $f$ and $g$ be polynomials over $\mathbb{F}_{q}$, each of which is coprime to $X$ and has no repeated root. Denote the order of $f$ and $g$ by $n$ and $m$, and the primitive defining set of $f$ and $g$ by $T_{f}$ and $T_{g}$, respectively. Then $f \mid g$ if and only if 
	$$i_{n,\mathrm{lcm}(n,m)}(T_{f}) \subseteq i_{m,\mathrm{lcm}(n,m)}(T_{g}).$$
\end{lemma}

Let $n$ be a positive integer coprime to the prime power $q$. For any $\gamma \in \mathbb{Z}/n\mathbb{Z}$, the $q$-cyclotomic coset modulo $n$ containing $\gamma$ is defined to be 
$$c_{n/q}(\gamma) = \{\gamma,\gamma q,\cdots,\gamma q^{\tau-1}\} \subseteq \mathbb{Z}/n\mathbb{Z},$$
where $\tau$ is the smallest positive integer such that $\gamma q^{\tau} = \gamma$. We refer to $\gamma$ as a representative of $c_{n/q}(\gamma)$ and $\tau$ as the size of $c_{n/q}(\gamma)$. And we denote by $\mathcal{C}_{n/q}$ the space of all $q$-cyclotomic cosets modulo $n$.

It is a well-known fact that there is an one-to-one correspondence between $\mathcal{C}_{n/q}$ and the set of all irreducible factors of $X^{n}-1$ over $\mathbb{F}_{q}$, which assigns to $c_{n/q}(\gamma) \in \mathcal{C}_{n/q}$ the irreducible polynomial
$$M_{\gamma}(X) = (X-\zeta_{n}^{\gamma})(X-\zeta_{n}^{\gamma q})\cdots(X-\zeta_{n}^{\gamma q^{\tau-1}}).$$
In other words, $c_{n/q}(\gamma)$ is the defining set modulo $n$ of $M_{\gamma}(X)$. Note that it is primitive if and only if $\gamma$ and $n$ are coprime.

Throughout this manuscript, by a binomial we mean a polynomial of the form $X^{n}-\lambda$, where $n \in \mathbb{N}^{+}$ and $\lambda \neq 0$.

\section{Binomials over finite fields}
\subsection{An equivalent characterization of binomials via defining sets}
In this subsection we give an equivalent characterization of binomials over $\mathbb{F}_{q}$ in terms of the defining sets of their radicals. As for a squarefree binomial $f(X)$ and a nonzero element $a$ of $\mathbb{F}_{q}$, $af(X)$ is also a binomial, and $f(X)$ and $af(X)$ share the same defining set, without losing generality we assume that the binomials are monic in this subsection.

Let $f(X) = X^{n}-\lambda$ be a binomial over $\mathbb{F}_{q}$. It is well-known that $f(X)$ has repeated root if and only if the characteristic $p$ of $\mathbb{F}_{q}$ divides $n$. If it is this case, setting $n=p^{v}n^{\prime}$, where $v = v_{p}(n)$ and $\mathrm{gcd}(p,n^{\prime})=1$, and $\lambda^{\prime}=\lambda^{t}$ where $t$ is the unique integer such that $1 \leq t \leq \mathrm{ord}(\lambda)-1$ and $tp^{v} \equiv 1 \pmod{\mathrm{ord}(\lambda)}$, then $\lambda^{\prime p^{v}} = \lambda$ and $f(X)$ can be expressed as
$$f(X) = X^{n}-\lambda = (X^{n^{\prime}}-\lambda^{\prime})^{p^{v}}.$$
Moreover, every root of $f(X)$ has multiplicity $p^{v}$.

By the above argument, to understand the binomials over $\mathbb{F}_{q}$ one need only concentrate on the case of squarefree binomials. Theorem \ref{thm 1} characterizes a squarefree binomials via its defining set, which generalizes Theorem 5.1 in \cite{Zhu 2}. A subset $E$ of $\mathbb{Z}/m\mathbb{Z}$ is said to be an equal-difference subset if it can be expressed as
$$E = \{\gamma,\gamma+d,\cdots,\gamma+(\frac{m}{d}-1)d\},$$
where $d$ is a positive divisor of $m$, called the common difference of $E$.

\begin{theorem}\label{thm 1}
	Let $f(X)$ be a nonconstant squarefree polynomial over $\mathbb{F}_{q}$ that is coprime to $X$, and let $m$ be the order of $f(X)$. Then $f(X)$ is a binomial if and only if its defining set $T_{f}$ is an equal-difference subset of $\mathbb{Z}/m\mathbb{Z}$.
\end{theorem}

\begin{proof}
	First assume that the defining set
	$$T_{f} = \{\gamma,\gamma+d,\cdots,\gamma+(n-1)d\}$$
	of $f(X)$ is of equal difference, where $n = \frac{m}{d}$. Expand $f(X)$ as
	$$f(X)= (X-\zeta_{m}^{\gamma})(X-\zeta_{m}^{\gamma+d})\cdots(X-\zeta_{m}^{\gamma+(n-1)d}) = X^{n} + a_{1}X^{n-1} + \cdots + a_{n-1}X +a_{n}.$$
	We prove by induction that $a_{1} = a_{2} = a_{n-1} = 0$. First, it can be computed directly that
	$$a_{1} = -(\zeta_{m}^{\gamma} + \cdots + \zeta_{m}^{\gamma+(n-1)d}) = -\zeta_{m}^{\gamma}\cdot \dfrac{\zeta_{m}^{m}-1}{\zeta_{m}^{d}-1} = 0,$$
	as $\zeta_{m}^{d}-1 \neq 0$. Now suppose that $a_{1} = \cdots = a_{k} = 0$, where $1 \leq k < n-1$. Note that for any $1 \leq \ell < n$,
	$$a_{\ell} = (-1)^{\ell}\sum_{0 \leq j_{1} < \cdots < \cdots j_{\ell} \leq n-1}\zeta_{m}^{\gamma+j_{1}d}\cdots\zeta_{m}^{\gamma+j_{\ell}d} = (-1)^{\ell}\dfrac{1}{(\ell)!}\zeta_{m}^{\ell\gamma}\cdot\sideset{}{^{\prime}}\sum_{j_{1},\cdots,j_{\ell}}\zeta_{m}^{j_{1}d}\cdots\zeta_{m}^{j_{\ell}d},$$
	where the sum $\sideset{}{^{\prime}}\sum\limits_{j_{1},\cdots,j_{\ell}}$ takes all $\ell$-tuples $(j_{1},\cdots,j_{\ell}) \in \{0,1,\cdots,n-1\}^{\ell}$ such that $j_{1},\cdots,j_{\ell}$ are pairwise distinct. As $a_{k}= (-1)^{k}\dfrac{1}{k!}\zeta_{m}^{k\gamma}\cdot\sideset{}{^{\prime}}\sum\limits_{j_{1},\cdots,j_{k}}\zeta_{m}^{j_{1}d}\cdots\zeta_{m}^{j_{k}d} = 0$, therefore
	\begin{align*}
		\sideset{}{^{\prime}}\sum\limits_{j_{1},\cdots,j_{k+1}}\zeta_{m}^{j_{1}d}\cdots\zeta_{m}^{j_{k+1}d} &= \sum_{j_{1}=0}^{n-1}\zeta_{m}^{j_{1}d}(\sideset{}{^{\prime}}\sum\limits_{\substack{j_{2},\cdots,j_{k+1}\\ j_{i} \neq j_{1}}}\zeta_{m}^{j_{2}d}\cdots\zeta_{m}^{j_{k+1}d})\\
		&= \sum_{j_{1}=0}^{n-1}\zeta_{m}^{j_{1}d}(-k\cdot\sideset{}{^{\prime}}\sum\limits_{\substack{j_{2},\cdots,j_{k}\\ j_{i} \neq j_{1}}}\zeta_{m}^{j_{1}d}\zeta_{m}^{j_{2}d}\cdots\zeta_{m}^{j_{k}d})\\
		&= -k\cdot \sum_{j_{1}=0}^{n-1}\zeta_{m}^{2j_{1}d}(\sideset{}{^{\prime}}\sum\limits_{\substack{j_{2},\cdots,j_{k}\\ j_{i} \neq j_{1}}}\zeta_{m}^{j_{2}d}\cdots\zeta_{m}^{j_{k}d}),
	\end{align*}
	where the second equality holds because $\sideset{}{^{\prime}}\sum\limits_{j_{1},\cdots,j_{k}}\zeta_{m}^{j_{1}d}\cdots\zeta_{m}^{j_{k}d} = 0$ and each term $\zeta_{m}^{j_{1}d}\zeta_{m}^{j_{2}d}\cdots\zeta_{m}^{j_{k}d}$ for a fixed $j_{1}$ appears repeatedly $k$ times in $\sideset{}{^{\prime}}\sum\limits_{j_{1},\cdots,j_{k}}\zeta_{m}^{j_{1}d}\cdots\zeta_{m}^{j_{k}d}$. Applying the same argument successively with $a_{k-1}= \cdots = a_{1}=0$ yields that 
	$$\sideset{}{^{\prime}}\sum\limits_{j_{1},\cdots,j_{k+1}}\zeta_{m}^{j_{1}d}\cdots\zeta_{m}^{j_{k+1}d} = (-1)^{k}k!\cdot \sum_{j_{1}=0}^{n-1}\zeta_{m}^{(k+1)j_{1}d},$$
	which indicates that 
	$$a_{k+1}= -\frac{1}{k+1}\zeta_{m}^{(k+1)\gamma}\cdot\sum\limits_{j=0}^{n-1}\zeta_{m}^{(k+1)jd}.$$
	Since $k+1 < n$, then $\zeta_{m}^{(k+1)d} \neq 1$ and 
	$$a_{k+1}= -\frac{1}{k+1}\zeta_{m}^{(k+1)\gamma}\cdot\dfrac{\zeta_{m}^{(k+1)n d}-1}{\zeta_{m}^{(k+1)d}-1} = 0.$$
	By induction it holds that $a_{1} = a_{2} = a_{n-1} = 0$.
	
	Conversely, assume that $f(X) = X^{n}-\lambda \in \mathbb{F}_{q}[X]$. Since $f(X)$ is squarefree, $n$ is coprime to $q$. Denote by $r$ the order of $\lambda$. There is a primitive $nr$-th root of unity, say $\zeta_{nr}^{\delta}$ where $\delta$ is coprime to $nr$, lying in $\overline{\mathbb{F}}_{q}$ such that $(\zeta_{nr}^{\delta})^{n} = \lambda$. Then $f(X)$ can be completely factorized as
	\begin{align*}
		X^{n}-\lambda &= \zeta_{nr}^{\delta n}((\frac{X}{\zeta_{nr}^{\delta}})^{n}-1)\\
		&= \zeta_{nr}^{\delta n}(\frac{X}{\zeta_{nr}^{\delta}}-1)(\frac{X}{\zeta_{nr}^{\delta}}-\zeta_{n})\cdots(\frac{X}{\zeta_{nr}^{\delta}}-\zeta_{n}^{n-1})\\
		&=(X-\zeta_{nr}^{\delta})(X-\zeta_{nr}^{\delta+r})\cdots(X-\zeta_{nr}^{\delta+(n-1)r}).
	\end{align*}
	Clearly the order $m$ of $f(X)$ divides $nr$, and the defining set of $f(X)$ modulo $nr$ are given by
	$$i_{m,nr}(T_{f}) = \{\delta,\delta+r,\cdots,\delta+(n-1)r\},$$
	which is an equal-difference subset of $\mathbb{Z}/nr\mathbb{Z}$. Now the conclusion follows directly from the next lemma.
\end{proof}

\begin{lemma}
	Let $n$ and $m$ be any positive integers such that $n \mid m$. Then a subset $E$ of $\mathbb{Z}/n\mathbb{Z}$ is of equal difference if and only if its image $i_{n,m}(E)$ in $\mathbb{Z}/m\mathbb{Z}$ is of equal difference.
\end{lemma}

\begin{proof}
	The assertion is clear from the definition of equal-difference subset.
\end{proof}

In particular, if $f(X)$ is irreducible, then its defining set is a $q$-cyclotomic coset $c_{m/q}(\gamma)$ modulo $m$. Theorem \ref{thm 1} indicates that $f(X)$ is a binomial if and only if $c_{m/q}(\gamma)$ is an equal-difference coset. 

On the other hand, by the proof of Theorem \ref{thm 1} any equal-difference subset
$$T = \{\gamma,\gamma+d,\cdots,\gamma+(\frac{m}{d}-1)d\} \subseteq \mathbb{Z}/m\mathbb{Z},$$
where $d \mid m$, induces a binomial
\begin{align*}
	f_{T}(X) &= \prod_{i=0}^{\frac{m}{d}-1}(X - \zeta_{m}^{\gamma+id})\\
	&=X^{\frac{m}{d}} + (-1)^{\frac{m}{d}}\zeta_{m}^{\frac{\gamma m}{d}}\zeta_{m}^{d+\cdots+(\frac{m}{d}-1)d}\\
	&=X^{\frac{m}{d}} - \zeta_{m}^{\frac{\gamma m}{d}}.
\end{align*}
Here the last equality holds since if $m$ is even then
$$(-1)^{\frac{m}{d}}\zeta_{m}^{d+\cdots+(\frac{m}{d}-1)d} = (-1)^{\frac{m}{d}}\zeta_{m}^{\frac{m}{2}\cdot(\frac{m}{d}-1)} = (-1)^{\frac{m}{d}}\cdot(-1)^{\frac{m}{d}-1}=-1;$$
otherwise, if $m$ is odd then so is $\frac{m}{d}$ and consequently
$$(-1)^{\frac{m}{d}}\zeta_{m}^{d+\cdots+(\frac{m}{d}-1)d} = (-1)^{\frac{m}{d}}\zeta_{m}^{m\cdot\frac{1}{2}(\frac{m}{d}-1)} = (-1)^{\frac{m}{d}}=-1.$$
Notice that in general $f_{T}(X)$ is over the extension field $\mathbb{F}_{q}(\zeta_{m})$ of $\mathbb{F}_{q}$. The criterion for $f_{T}(X)$ being defined over $\mathbb{F}_{q}$ is given below.

\begin{lemma}\label{lem 1}
	The polynomial $f_{T}(X)$ is defined over $\mathbb{F}_{q}$ if and only if 
	$$\gamma q \equiv \gamma \pmod{d}.$$
	If it is this case, then the set $T$ is called an equal-difference subset defined over $\mathbb{F}_{q}$.
\end{lemma}

\begin{proof}
	Since $f_{T}(X) = X^{\frac{m}{d}} - \zeta_{m}^{\frac{\gamma m}{d}}$, it is defined over $\mathbb{F}_{q}$ if and only if $\zeta_{m}^{\frac{\gamma m}{d}} \in \mathbb{F}_{q}$, which amounts to that $\zeta_{m}^{\frac{\gamma m}{d}}$ is stable under the Frobenius automorphism $\mathrm{Frob}_{q}$ over $\mathbb{F}_{q}$, i.e.,
	\begin{equation}\label{eq 1}
		(\zeta_{m}^{\frac{\gamma m}{d}})^{\mathrm{Frob}_{q}} = \zeta_{m}^{q\cdot\frac{\gamma m}{d}} = \zeta_{m}^{\frac{\gamma m}{d}}.
	\end{equation}
	As $\zeta_{m}$ is a primitive $m$-th root of unity, the identity \eqref{eq 1} is valid if and only if the congruence
	$$\frac{q\gamma m}{d} \equiv \frac{\gamma m}{d} \pmod{m},$$
	or equivalently, $\gamma q \equiv \gamma \pmod{d}$ holds.
\end{proof}

Now we turn to the general case. Let $f(X) \in \mathbb{F}_{q}[X]$, with the factorization
$$f(X) = f_{1}(X)^{e_{1}}\cdots f_{t}(X)^{e_{t}}$$
of $f(X)$ into powers of distinct irreducible polynomials. The radical of $f(X)$ is defined to be 
$$\mathrm{rad}(f) = f_{1}(X)\cdots f_{t}(X).$$

\begin{theorem}
	Let $f(X)$ be a nonconstant polynomial over $\mathbb{F}_{q}$ that is coprime to $X$, and let $m$ be the order of the radical $\mathrm{rad}(f)$ of $f(X)$. Then $f(X)$ is a binomial if and only if there is a nonnegative integer $v$ such that $f(X) = \mathrm{rad}(f)^{p^{v}}$ and the defining set $T_{\mathrm{rad}(f)}$ of $\mathrm{rad}(f)$ is an equal-difference subset of $\mathbb{Z}/m\mathbb{Z}$.
\end{theorem}

\subsection{Irreducible factorizations of binomials over finite fields}
This subsection is devoted to a brief reminder of the formula for the irreducible factorizations of binomials over finite fields, which serves as a preparation for Section \ref{sec 1}. For the detailed treatment and the proofs we refer the readers to \cite{Zhu3}.

\subsubsection{Explicit representatives and sizes of cyclotomic cosets}
Let $\ell$ be a prime number different from $p$, and $n$ be a positive integer which is divisible by neither $p$ nor $\ell$. Let $c_{n/q}(\gamma)$ be a $q$-cyclotomic coset modulo $n$, with size $\tau_{\gamma} = | c_{n/q}(\gamma) |$. For any positive integer $u$, we set up the following notations.
\begin{itemize}
	\item[(1)] If $\ell \mid q^{\tau_{\gamma}}-1$, then set
	\begin{equation*}
		v_{m}^{\ell^{u}}(\tau_{\gamma}) = \left\{
		\begin{array}{lcl}
			\max\{\mathrm{min}\{u-m-2,v_{\ell}(q^{\tau_{\gamma}}+1)-1\},0\}, \ \mathrm{if} \ \ell=2 \ \mathrm{and} \ q^{\tau_{\gamma}} \equiv 3 \pmod{4};\\
			\max\{\mathrm{min}\{u-m-1,v_{\ell}(q^{\tau_{\gamma}}-1)-1\},0\}, \ \mathrm{otherwise},
		\end{array} \right.
	\end{equation*}
	\begin{equation*}
		\mu_{m,i,\mathbf{t}}^{\ell^{u},n}(\gamma) = \left\{
		\begin{array}{lcl}
			\gamma+n\cdot(U_{m,i}^{\ell,n}(\gamma)+\sum\limits_{j=1}^{v_{m}^{\ell^{u}}(\tau_{\gamma})}t_{j}\cdot \ell^{m+j+1}), \ \mathrm{if} \ \ell=2 \ \mathrm{and} \ q^{\tau_{\gamma}} \equiv 3 \pmod{4};\\
			\gamma+n\cdot(U_{m,i}^{\ell,n}(\gamma)+\sum\limits_{j=1}^{v_{m}^{\ell^{u}}(\tau_{\gamma})}t_{j}\cdot \ell^{m+j}), \ \mathrm{otherwise},
		\end{array} \right.
	\end{equation*}
	where $0 \leq m \leq u$, $i \in [1,\ell-1]$, $\mathbf{t} \in [0,\ell-1]^{v_{m}^{\ell^{u}}(\tau_{\gamma})}$, and $U_{m,i}^{\ell,n}(\gamma)$ is the generator of degree $m$ of the $\ell$-adic $q$-cyclotomic system $\mathcal{PC}_{n/q}^{\ell}(\gamma)$. And denote by $\Sigma_{\ell^{u}n/n}(\gamma)$ the set
	$$\Sigma_{\ell^{u}n/n}(\gamma) = \{\mu_{m,i,\mathbf{t}}^{\ell^{u},n}(\gamma) \ | \ (m,i,\mathbf{t}) \in \bigsqcup_{m=0}^{u-1}(\{m\}\times[1,\ell-1]\times[0,\ell-1]^{v_{m}^{\ell^{u}}(\tau_{\gamma})})\sqcup\{(u,1,0)\}\}.$$
	\item[(2)] If $\ell \nmid q^{\tau_{\gamma}}-1$, then set
	$$v_{m}^{\ell^{u}}(\tau_{\gamma}) = \max\{\mathrm{min}\{u-m-1,v_{\ell}(q^{\tau_{\gamma}\cdot\mathrm{ord}_{\ell}(q^{\tau_{\gamma}})}-1)-1\},0\},$$
	$$\mu_{m,i,\mathbf{t}}^{\ell^{u},n}(\gamma) = \gamma+n\cdot(U_{m,i}^{\ell,n}(\gamma)+\sum\limits_{j=1}^{v_{m}^{\ell^{u}}(\tau_{\gamma})}t_{j}\cdot \ell^{m+j}),$$
	where $0 \leq m \leq u$, $i \in [1,h]$ for $h= \frac{\ell-1}{\mathrm{ord}_{\ell}(q^{\tau_{\gamma}})}$, $\mathbf{t} \in [0,\ell-1]^{v_{m}^{\ell^{u}}(\tau_{\gamma})}$, and $U_{m,i}^{\ell,n}(\gamma)$ is the generator of degree $m$ of the $\ell$-adic $q$-cyclotomic system $\mathcal{PC}_{n/q}^{\ell}(\gamma)$. And denote by $\Sigma_{\ell^{u}n/n}(\gamma)$ the set
	$$\Sigma_{\ell^{u}n/n}(\gamma) = \{\mu_{m,i,\mathbf{t}}^{\ell^{u},n}(\gamma) \ | \ (m,i,\mathbf{t}) \in \bigsqcup_{m=0}^{u-1}(\{m\}\times[1,h]\times[0,\ell-1]^{v_{m}^{\ell^{u}}(\tau_{\gamma})})\sqcup\{(u,1,0)\}\}.$$
\end{itemize}

\begin{lemma}\label{lem 4}
	The set $\Sigma_{\ell^{u}n/n}(\gamma)$ is a full set of representatives of the $q$-cyclotomic cosets modulo $\ell^{u}n$ which are contained in the preimage $\pi_{\ell^{u}n/n}^{-1}(c_{n/q}(\gamma))$ of $c_{n/q}(\gamma)$ along the canonical projection $\pi_{\ell^{u}n/n}: \mathbb{Z}/\ell^{u}n\mathbb{Z} \rightarrow \mathbb{Z}/n\mathbb{Z}$.
\end{lemma}

Let $p_{1} < \cdots < p_{s}$ be prime numbers not dividing $q$ and $n$, and $e_{1},\cdots,e_{s}$ be positive integers. We define the set $\mathcal{CR}_{p_{1}^{e_{1}}\cdots p_{s}^{e_{s}}n/q}^{j}(\gamma)$, $0 \leq j \leq s$, inductively by
\begin{itemize}
	\item[(1)] $\mathcal{CR}_{p_{1}^{e_{1}}\cdots p_{s}^{e_{s}}n/q}^{0}(\gamma) = \{\gamma\}$;
	\item[(2)] if $j$ is such that $0 \leq j \leq s-1$, then $\mathcal{CR}_{p_{1}^{e_{1}}\cdots p_{s}^{e_{s}}n/q}^{j+1}(\gamma)$ is settled to be
	$$\mathcal{CR}_{p_{1}^{e_{1}}\cdots p_{s}^{e_{s}}n/q}^{j+1}(\gamma) = \bigsqcup_{\eta \in \mathcal{CR}_{p_{1}^{e_{1}}\cdots p_{s}^{e_{s}}n/q}^{j}(\gamma)}\Sigma_{p_{j+1}^{e_{j+1}}p_{1}^{e_{1}}\cdots p_{j}^{e_{j}}n/p_{1}^{e_{1}}\cdots p_{j}^{e_{j}}n}(\eta).$$
\end{itemize}
And we write $\mathcal{CR}_{p_{1}^{e_{1}}\cdots p_{s}^{e_{s}}n/q}(\gamma) = \mathcal{CR}_{p_{1}^{e_{1}}\cdots p_{s}^{e_{s}}n/q}^{s}(\gamma)$. Then applying Lemma \ref{lem 4} successively leads to the following theorem.

\begin{theorem}
	Let $p_{1} < \cdots < p_{s}$ be prime numbers not dividing $q$ and $n$, and $e_{1},\cdots,e_{s}$ are positive integers. Let $c_{n/q}(\gamma)$ be a $q$-cyclotomic coset modulo $n$, with size $\tau_{\gamma}$. Then the set $\mathcal{CR}_{p_{1}^{e_{1}}\cdots p_{s}^{e_{s}}n/q}(\gamma)$ is a full set of representatives of the $q$-cyclotomic cosets modulo $p_{1}^{e_{1}}\cdots p_{s}^{e_{s}}n$ which are contained in $\pi_{p_{1}^{e_{1}}\cdots p_{s}^{e_{s}}n/n}^{-1}(c_{n/q}(\gamma))$.
\end{theorem}

To present the elements in $\mathcal{CR}_{p_{1}^{e_{1}}\cdots p_{s}^{e_{s}}n/q}(\gamma)$ concretely and determine the sizes of the corresponding cyclotomic cosets, we introduce the following definitions. For any pair $\mathbf{a} = (a_{1},\cdots,a_{k})$ and a nonnegative integer $k^{\prime} \leq k$, we write 
\begin{equation*}
	\mathbf{a}_{\leq k^{\prime}} = \left\{
	\begin{array}{lcl}
		\ast, \ \mathrm{if} \ k^{\prime}=0;\\
		(a_{1},\cdots,a_{k^{\prime}}), \ \mathrm{if} \ k^{\prime} > 0,
	\end{array} \right.
\end{equation*}
where $\ast$ is the empty symbol.

For $j=0$ we set
$$R_{\ast,\ast,\ast}^{0}(\gamma) = \gamma \ \mathrm{and} \ S_{\ast}^{0}(\gamma) = \tau_{\gamma}.$$

Suppose that $j$ is such that $0 \leq j \leq s-1$. Let $\mathbf{m} = (m_{1},\cdots,m_{j},m_{j+1})$ be a pair of integers where $0 \leq m_{k} \leq e_{k}$ for $k=0,\cdots,j+1$, $\mathbf{i} = (i_{1},\cdots,i_{j},i_{j+1})$ be a pair of integers which satisfies that for any $0 \leq k \leq j+1$
\begin{equation*}
	\left\{
	\begin{array}{lcl}
		i_{k} = 1, \ \mathrm{if} \ m_{k}=e_{k};\\
		1 \leq i_{k} \leq p_{k}-1, \ \mathrm{if} \ m_{k} < e_{k} \ \mathrm{and} \ p_{k}\mid q^{S_{\mathbf{m}_{\leq k-1}}^{k-1}(\gamma)}-1;\\
		1 \leq i_{k} \leq \frac{p_{k}-1}{\mathrm{ord}_{p_{k}}(q^{S_{\mathbf{m}_{\leq k-1}}^{k-1}(\gamma)})}, \ \mathrm{if} \ m_{k} < e_{k} \ \mathrm{and} \ p_{k}\nmid q^{S_{\mathbf{m}_{\leq k-1}}^{k-1}(\gamma)}-1,
	\end{array} \right.
\end{equation*}
and $\mathbf{T} = (\mathbf{t}_{1},\cdots,\mathbf{t}_{j},\mathbf{t}_{j+1})$ be a pair consisting of $\mathbf{t}_{k} \in [0,p_{k}-1]^{v_{m_{k}}^{p_{k}^{e_{k}}}(S_{\mathbf{m}_{\leq k-1}}^{k-1}(\gamma))}$. Then for $j+1$ we set
$$R_{\mathbf{m},\mathbf{i},\mathbf{T}}^{j+1}(\gamma) = \mu_{m_{j+1},i_{j+1},\mathbf{t}_{j+1}}^{p_{j+1}^{e_{j+1}},p_{1}^{e_{1}}\cdots p_{j}^{e_{j}}}(R_{\mathbf{m}_{\leq j},\mathbf{i}_{\leq j},\mathbf{T}_{\leq j}}^{j}(\gamma))$$
and $S_{\mathbf{m}}^{j+1}(\gamma)$ to be such that
\begin{itemize}
	\item[(1)] If either $p_{j+1}$ is an odd prime dividing $q^{S_{\mathbf{m}_{\leq j}}^{j}(\gamma)}-1$ or $p_{j+1}=2$ and $q^{S_{\mathbf{m}_{\leq j}}^{j}(\gamma)} \equiv 1 \pmod{4}$, then
	$$S_{\mathbf{m}}^{j+1}(\gamma) = p_{j+1}^{\mathrm{max}\{e_{j+1}-m_{j+1}-v_{p_{j+1}}(q^{S_{\mathbf{m}_{\leq j}}^{j}(\gamma)}-1),0\}}\cdot S_{\mathbf{m}_{\leq j}}^{j}(\gamma);$$
	\item[(2)] If $p_{j+1}=2$ and $q^{S_{\mathbf{m}_{\leq j}}^{j}(\gamma)} \equiv 3 \pmod{4}$, then
	\begin{equation*}
		S_{\mathbf{m}}^{j+1}(\gamma) = \left\{
		\begin{array}{lcl}
			2^{\mathrm{max}\{e_{j+1}-m_{j+1}-v_{2}(q^{S_{\mathbf{m}_{\leq j}}^{j}(\gamma)}+1),1\}}\cdot S_{\mathbf{m}_{\leq j}}^{j}(\gamma), \ \mathrm{if} \ 0 \leq m_{j+1} \leq e_{j+1}-2;\\
			S_{\mathbf{m}_{\leq j}}^{j}(\gamma), \ \mathrm{if} \ m_{j+1}= e_{j+1}-1 \ \mathrm{or} \ e_{j+1};
		\end{array} \right.
	\end{equation*}
	\item[(3)] If $p_{j+1} \nmid q^{S_{\mathbf{m}_{\leq j}}^{j}(\gamma)}-1$, then
	\begin{equation*}
		S_{\mathbf{m}}^{j+1}(\gamma) = \left\{
		\begin{array}{lcl}
			p_{j+1}^{\mathrm{max}}\cdot\mathrm{ord}_{p_{j+1}}(q^{S_{\mathbf{m}_{\leq j}}^{j}(\gamma)})\cdot S_{\mathbf{m}_{\leq j}}^{j}(\gamma), \ \mathrm{if} \ 0 \leq m_{j+1} \leq e_{j+1}-1;\\
			S_{\mathbf{m}_{\leq j}}^{j}(\gamma), \ \mathrm{if} \ m_{j+1}= e_{j+1}.
		\end{array} \right.
	\end{equation*}
	where $\mathrm{max} = \mathrm{max}\{e_{j+1}-m_{j+1}-v_{p_{j+1}}(q^{S_{\mathbf{m}_{\leq j}}^{j}(\gamma)\cdot \mathrm{ord}_{p_{j+1}}(q^{S_{\mathbf{m}_{\leq j}}^{j}(\gamma)})}-1),0\}$.
\end{itemize}
In particular, for $j=s$ we omit the superscript and simply write $R_{\mathbf{m},\mathbf{i},\mathbf{T}}(\gamma) = R_{\mathbf{m},\mathbf{i},\mathbf{T}}^{s}(\gamma)$ and $S_{\mathbf{m}}(\gamma) = S_{\mathbf{m}}^{s}(\gamma)$. We denote by $I_{p_{1}^{e_{1}}\cdots p_{s}^{e_{s}}n/q}(\gamma)$ the set of pairs $(\mathbf{m},\mathbf{i},\mathbf{T})$ which satisfy that
\begin{itemize}
	\item[(1)] $\mathbf{m} = (m_{1},\cdots,m_{s})$ is a pair of integers where $0 \leq m_{k} \leq e_{k}$ for $k=1,\cdots,s$;
	\item[(2)] $\mathbf{i} = (i_{1},\cdots,i_{s})$ is a pair of integers such that for any $1 \leq k \leq s$
	\begin{equation*}
		\left\{
		\begin{array}{lcl}
			i_{k} = 1, \ \mathrm{if} \ m_{k}=e_{k};\\
			1 \leq i_{k} \leq p_{k}-1, \ \mathrm{if} \ m_{k} < e_{k} \ \mathrm{and} \ p_{k}\mid q^{S_{\mathbf{m}_{\leq k-1}}^{k-1}(\gamma)}-1;\\
			1 \leq i_{k} \leq \frac{p_{k}-1}{\mathrm{ord}_{p_{k}}(q^{S_{\mathbf{m}_{\leq k-1}}^{k-1}(\gamma)})}, \ \mathrm{if} \ m_{k} < e_{k} \ \mathrm{and} \ p_{k}\nmid q^{S_{\mathbf{m}_{\leq k-1}}^{k-1}(\gamma)}-1,
		\end{array} \right.
	\end{equation*}
	\item[(3)] $\mathbf{T} = (\mathbf{t}_{1},\cdots,\mathbf{t}_{s})$ consists of $\mathbf{t}_{k} \in [0,p_{k}-1]^{v_{m_{k}}^{p_{k}^{e_{k}}}(S_{\mathbf{m}_{\leq k-1}}^{k-1}(\gamma))}$.
\end{itemize}

\begin{theorem}\label{thm 3}
	The map
	$$I_{p_{1}^{e_{1}}\cdots p_{s}^{e_{s}}n/q}(\gamma) \rightarrow \mathcal{CR}_{p_{1}^{e_{1}}\cdots p_{s}^{e_{s}}n/q}(\gamma); \ (\mathbf{m},\mathbf{i},\mathbf{T}) \mapsto R_{\mathbf{m},\mathbf{i},\mathbf{T}}(\gamma)$$
	is a bijection. For any $(\mathbf{m},\mathbf{i},\mathbf{T}) \in I_{p_{1}^{e_{1}}\cdots p_{s}^{e_{s}}n/q}(\gamma)$, the size of the coset $c_{p_{1}^{e_{1}}\cdots p_{s}^{e_{s}}n/q}(R_{\mathbf{m},\mathbf{i},\mathbf{T}}(\gamma))$ is 
	$$|c_{p_{1}^{e_{1}}\cdots p_{s}^{e_{s}}n/q}(R_{\mathbf{m},\mathbf{i},\mathbf{T}}(\gamma))| = S_{\mathbf{m}}(\gamma).$$
\end{theorem}

For a fixed $\mathbf{m}_{0} = (m_{0,1}\cdots,m_{0,s})$ which satisfies $0 \leq m_{0,k} \leq e_{k}$ for all $k=1,\cdots,s$, we define a subset $\mathcal{CR}_{p_{1}^{e_{1}}\cdots p_{s}^{e_{s}}n/q}^{\mathbf{m}_{0}}(\gamma)$ of $\mathcal{CR}_{p_{1}^{e_{1}}\cdots p_{s}^{e_{s}}n/q}(\gamma)$ by
$$\mathcal{CR}_{p_{1}^{e_{1}}\cdots p_{s}^{e_{s}}n/q}^{\mathbf{m}_{0}}(\gamma) = \{R_{\mathbf{m},\mathbf{i},\mathbf{T}}(\gamma) \in \mathcal{CR}_{p_{1}^{e_{1}}\cdots p_{s}^{e_{s}}n/q}(\gamma) \ | \ \mathbf{m}=\mathbf{m}_{0}\}.$$
Then the set $\mathcal{CR}_{p_{1}^{e_{1}}\cdots p_{s}^{e_{s}}n/q}(\gamma)$ can be classified as
$$\mathcal{CR}_{p_{1}^{e_{1}}\cdots p_{s}^{e_{s}}n/q}(\gamma) = \bigsqcup_{\mathbf{m}\in[0,e_{1}]\times\cdots\times[0,e_{s}]}\mathcal{CR}_{p_{1}^{e_{1}}\cdots p_{s}^{e_{s}}n/q}^{\mathbf{m}}(\gamma).$$
This classification is natural as Theorem \ref{thm 1} yields that the elements lying in the same subset $\mathcal{CR}_{p_{1}^{e_{1}}\cdots p_{s}^{e_{s}}n/q}^{\mathbf{m}}(\gamma)$ give rise to cyclotomic cosets with the same size $S_{\mathbf{m}}(\gamma)$.

\subsubsection{Irreducible factorizations of binomials over finite fields}\label{Sec 2}
Let $X^{N}-\lambda$ be a binomial over $\mathbb{F}_{q}$, where $N \in \mathbb{N}^{+}$ and $\lambda \in \mathbb{F}_{q}^{\ast}$, and let $r$ be the order of $\lambda$. Write $N = p^{v}n$, where $v = v_{p}(N) \geq 0$ and $n$ is coprime to $q$. Then $X^{N}-\lambda$ can be expressed as
$$X^{N}-\lambda = (X^{n}-\lambda^{\prime})^{p^{v}},$$
where $\lambda^{\prime} = \lambda^{t}$ for t satisfying $1 \leq t \leq r-1$ and $tp^{v} \equiv 1 \pmod{r}$. Notice that the order of $\lambda^{\prime}$ is also $r$. Then there is a primitive $nr$-th root $\zeta_{nr}^{\delta}$ of unity such that $\zeta_{nr}^{\delta n}=\lambda^{\prime}$, where $\delta$ is a positive integer coprime to $nr$. It follows from the proof of Theorem \ref{thm 1} that $X^{n^{\prime}}-\lambda^{\prime}$ has the defining set modulo $nr$ given by
$$i_{m,nr}(T) = \{\delta,\delta+r,\cdots,\delta+r(n-1)\} = \delta + r\cdot\mathbb{Z}/nr\mathbb{Z}.$$
Here $m$ denotes the order of $X^{n^{\prime}}-\lambda^{\prime}$. As $r \mid q-1$, the $q$-cyclotomic coset $c_{r/q}(\delta)$ modulo $r$ consists of one element $\delta$, which implies that 
$$i_{m,nr}(T) = \pi_{nr/r}^{-1}(c_{r/q}(\delta)).$$

Assume that $\mathrm{rad}(\mathrm{gcd}(n,r))=r_{1}\cdots r_{t}$ where $r_{1} < \cdots < r_{t}$ are prime numbers. We write $r = r_{1}^{a_{1}}\cdots r_{t}^{a_{t}}r^{\prime}$ and $n = r_{1}^{b_{1}}\cdots r_{t}^{b_{t}}p_{1}^{e_{1}}\cdots p_{s}^{e_{s}}$, where $r^{\prime}$ is coprime to $n$ and $p_{1} < \cdots < p_{s}$ are prime numbers different from $r_{1},\cdots,r_{t}$. The factor $r_{1}^{b_{1}}\cdots r_{t}^{b_{t}}$ of $n$ is denoted by $n_{1}$. We define the following notations:
\begin{itemize}
	\item[(1)] \begin{equation*}
		u_{1} = u_{1}(n,r) =  \left\{
		\begin{array}{lcl}
			\mathrm{min}\{b_{1}-1,v_{r_{1}}(q+1)-a_{1}\}, \quad \mathrm{if} \ r_{1}=2 \ \mathrm{and} \ q \equiv 3 \pmod{4};\\
			\mathrm{min}\{b_{1},v_{r_{1}}(q-1)-a_{1}\}, \quad \mathrm{otherwise}.
		\end{array} \right.
	\end{equation*}
	\item[(2)] $u_{j} = \mathrm{min}\{b_{j},v_{r_{j}}(q-1)-a_{j}\}$ for $j=2,\cdots,t$, and
	\item[(3)] $\Gamma = \Gamma(n,r) = [r_{1}-1]^{u_{1}}\times[r_{t}-1]^{u_{t}}$. In particular, if $u_{j}=0$ for some $j$ then $[r_{j}-1]^{u_{j}}$ is understood to be $\{0\}$.
\end{itemize}
Given any $\mathbf{x} = (\mathbf{x}_{1},\cdots,\mathbf{x}_{t}) \in \Gamma$, where $\mathbf{x}_{j} = (x_{j,1},\cdots,x_{j,u_{j}}) \in [0,r_{j}-1]^{u_{j}}$ for $1 \leq j \leq t$, we define 
\begin{small}
\begin{equation*}
\delta_{\mathbf{x}} =  \left\{
\begin{array}{lcl}
\delta+r\cdot\sum_{j_{1}=1}^{u_{1}}x_{1,j_{1}}\cdot r_{1}^{j_{1}}+rr_{1}^{a_{1}}\cdot\sum_{j_{2}=1}^{u_{2}}x_{1,j_{2}}\cdot r_{2}^{j_{2}-1}+\cdots+rr_{1}^{a_{1}}\cdots r_{t-1}^{a_{t-1}}\cdot\sum_{j_{t}=1}^{u_{t}}x_{1,j_{t}}\cdot r_{t}^{j_{t}-1},\\  ~~~~~~~~~~~~~~~~~~~~~~~~~~~~~~~~~~~~~~~~~~~~~~~~~~~~~~~~~~~~~~~~~~~~~~~~~~~~~~~~~~~~~~~\mathrm{if} \ r_{1}=2 \ \mathrm{and} \ q \equiv 3 \pmod{4};\\
\delta+r\cdot\sum_{j_{1}=1}^{u_{1}}x_{1,j_{1}}\cdot r_{1}^{j_{1}-1}+rr_{1}^{a_{1}}\cdot\sum_{j_{2}=1}^{u_{2}}x_{1,j_{2}}\cdot r_{2}^{j_{2}-1}+\cdots+rr_{1}^{a_{1}}\cdots r_{t-1}^{a_{t-1}}\cdot\sum_{j_{t}=1}^{u_{t}}x_{1,j_{t}}\cdot r_{t}^{j_{t}-1}, \\
~~~~~~~~~~~~~~~~~~~~~~~~~~~~~~~~~~~~~~~~~~~~~~~~~~~~~~~~~~~~~~~~~~~~~~~~~~~~~~~~~~~~~~~\mathrm{otherwise}.
\end{array} \right.
\end{equation*}
\end{small}

\begin{lemma}\label{lem 6}
	The set
	$$\{\delta_{\mathbf{x}} \ | \ \mathbf{x} \in \Gamma(n,r)\}$$
	is a full set of representatives of the $q$-cyclotomic cosets modulo $n_{1}r$ which are contained in $\pi_{n_{1}r/r}^{-1}(c_{r/q}(\delta))$. Moreover, all the cosets $c_{n_{1}r/q}(\delta_{\mathbf{x}})$ have the same size, given by
	\begin{equation*}
		|c_{n_{1}r/q}(\delta_{\mathbf{x}})| =  \left\{
		\begin{array}{lcl}
			r_{1}^{\mathrm{max}\{a_{1}+b_{1}-v_{r_{1}}(q+1),1\}}\cdot\prod\limits_{j=2}^{t}r_{j}^{\mathrm{max}\{a_{j}+b_{j}-v_{r_{j}}(q-1),0\}}, \ \mathrm{if} \ r_{1}=2 \ \mathrm{and} \ q \equiv 3 \pmod{4};\\
			\prod\limits_{j=1}^{t}r_{j}^{\mathrm{max}\{a_{j}+b_{j}-v_{r_{j}}(q-1),0\}}, \ \mathrm{otherwise}.
		\end{array} \right.
	\end{equation*}
\end{lemma}

Combining Theorem \ref{thm 3} and Lemma \ref{lem 6}, one obtains the deomposition of $T$ into a disjoint union of cyclotomic cosets.

\begin{theorem}
	The defining set $T$ of $X^{n}-\lambda^{\prime}$ can be decomposed into a disjoint union of $q$-cyclotomic cosets modulo $nr$ as
	$$T = \pi_{nr/r}^{-1}(c_{r/q}(\delta)) = \bigsqcup_{\mathbf{x} \in \Gamma(n,r)}\bigsqcup_{\mathbf{m} \in [0,e_{1}]\times\cdots\times[0,e_{s}]}\bigsqcup_{R_{\mathbf{m},\mathbf{i},\mathbf{T}}(\delta_{\mathbf{x}}) \in \mathcal{CR}_{nr/q}^{\mathbf{m}}(\delta_{\mathbf{x}})}c_{nr/q}(R_{\mathbf{m},\mathbf{i},\mathbf{T}}(\delta_{\mathbf{x}})).$$
\end{theorem}

For any $\mathbf{x} \in \Gamma(n,r)$, $\mathbf{m} \in [0,e_{1}]\times\cdots\times[0,e_{s}]$ and $R_{\mathbf{m},\mathbf{i},\mathbf{T}}(\delta_{\mathbf{x}}) \in \mathcal{CR}_{nr/q}^{\mathbf{m}}(\delta_{\mathbf{x}})$, it is trivial to verify that 
$$\mathrm{gcd}(R_{\mathbf{m},\mathbf{i},\mathbf{T}}(\delta_{\mathbf{x}}),nr) = p_{1}^{m_{1}}\cdots p_{s}^{m_{s}},$$
and 
$$\frac{nr}{\mathrm{gcd}(R_{\mathbf{m},\mathbf{i},\mathbf{T}}(\delta_{\mathbf{x}}),nr)} = rn_{1}p_{1}^{e_{1}-m_{1}}\cdots p_{s}^{e_{s}-m_{s}}.$$
Therefore we have
$$\mathrm{rad}(\frac{nr}{\mathrm{gcd}(R_{\mathbf{m},\mathbf{i},\mathbf{T}}(\delta_{\mathbf{x}}),nr)}) = \mathrm{rad}(r)\cdot p_{1}^{y_{1}}\cdots p_{s}^{y_{s}},$$
where $y_{j} = \mathrm{min}\{e_{j}-m_{j},1\}$ for all $j = 1,\cdots,s$. We set
\begin{equation*}
	\omega_{\mathbf{m}}(\delta_{\mathbf{x}}) = \left\{
	\begin{array}{lcl}
		2\mathrm{ord}_{\mathrm{rad}(r)\cdot p_{1}^{y_{1}}\cdots p_{s}^{y_{s}}}(q), \quad \mathrm{if} \ q^{\mathrm{ord}_{\mathrm{rad}(r)\cdot p_{1}^{y_{1}}\cdots p_{s}^{y_{s}}}(q)} \equiv 3 \pmod{4} \ \mathrm{and} \ 8 \mid rn_{1}p_{1}^{e_{1}-m_{1}}\cdots p_{s}^{e_{s}-m_{s}};\\
		\mathrm{ord}_{\mathrm{rad}(r)\cdot p_{1}^{y_{1}}\cdots p_{s}^{y_{s}}}(q), \quad \mathrm{otherwise},
	\end{array} \right.
\end{equation*}
and
$$d_{\mathbf{m}}(\delta_{\mathbf{x}}) = \frac{n\omega_{\mathbf{m}}(\delta_{\mathbf{x}})}{S_{\mathbf{m}}(\delta_{\mathbf{x}})}.$$
Then the irreducible factorization of $X^{N}-\lambda$ is given by the following theorem.

\begin{theorem}\label{thm 2}
	The irreducible factorization of $X^{N}-\lambda$ over $\mathbb{F}_{q}$ is given by
	$$X^{N}-\lambda = \prod_{\mathbf{x} \in \Gamma(n,r)}\prod_{\mathbf{m} \in [0,e_{1}]\times\cdots\times[0,e_{s}]}\prod_{R_{\mathbf{m},\mathbf{i},\mathbf{T}}(\delta_{\mathbf{x}}) \in \mathcal{CR}_{nr/q}^{\mathbf{m}}(\delta_{\mathbf{x}})}M_{R_{\mathbf{m},\mathbf{i},\mathbf{T}}(\delta_{\mathbf{x}})}(X)^{p^{v}},$$
	where 
	$$M_{R_{\mathbf{m},\mathbf{i},\mathbf{T}}(\delta_{\mathbf{x}})}(X) = \sum_{j=0}^{\omega_{\mathbf{m}}(\delta_{\mathbf{x}})-1}(-1)^{\omega_{\mathbf{m}}(\delta_{\mathbf{x}})-j}\left(\sum_{\substack{U \subseteq \{0,\cdots,\omega_{\mathbf{m}}(\delta_{\mathbf{x}})-1\}\\ |Y|=\omega_{\mathbf{m}}(\delta_{\mathbf{x}})-j}}\prod_{u \in U}\zeta_{d_{\mathbf{m}}(\delta_{\mathbf{x}})}^{R_{\mathbf{m},\mathbf{i},\mathbf{T}}(\delta_{\mathbf{x}})\cdot q^{u}}\right)X^{\frac{S_{\mathbf{m}}(\delta_{\mathbf{x}})}{\omega_{\mathbf{m}}(\delta_{\mathbf{x}})}\cdot j}.$$
\end{theorem}

\section{The minimal binomial multiple of $f(X)$}
\subsection{The definitions of minimal binomial multiple}
Let $f(X)$ be a nonconstant polynomial over $\mathbb{F}_{q}$ with a nonzero constant term. The order of $f(X)$ is a classical notion, which is defined to be the smallest positive integer $m$ such that $f(X) \mid X^{m}-1$. Recently, motivated by the study of primitive $k$-normal elements in finite fields, Mart\'{i}nez, Reis and Ribas \cite{Mart\'{i}nez} introduced the definition of freeness of binomials. The polynomial $f(X)$ is said to be free of binomials if it does not divide any binomial with degree less than $m = \mathrm{ord}(f)$. It is natural to raise the further problem to depict explicitly the binomial with the lowest degree that is divided by $f(X)$, which leads to the definition below.

\begin{definition}\label{def 1}
	Let $f(X)$ be a nonconstant polynomial over $\mathbb{F}_{q}$ with a nonzero constant term. Define the minimal binomial multiple of $f(X)$ to be the monic binomial $X^{n}-\lambda \in \mathbb{F}_{q}[X]$ which is of the lowest degree among the binomials over $\mathbb{F}_{q}$ divided by $f(X)$. The degree $n$ is called the minimal binomial order of $f(X)$, and is denoted by $\mathrm{ord}^{\mathrm{b}}(f) = n$.
\end{definition}

Since a polynomial with its constant term being zero divides no binomial, in the remaining part of this section, unless stating otherwise, the polynomials are assumed to be nonconstant and coprime to $X$.

Clearly there is a smallest positive integer $n$ such that there exists a monic binomial of degree $n$, say $X^{n}-\lambda$, that is divisible by $f(X)$. If $X^{n}-\lambda^{\prime}$ is also divide by $f(X)$, then
$$f(X) \mid X^{k}-\lambda^{\prime} - (X^{k}-\lambda) = \lambda-\lambda^{\prime}.$$
Since $f(X)$ is nonconstant, we have $\lambda = \lambda^{\prime}$, which indicates that the monic binomial with the lowest degree to be divided by $f(X)$ is unique. Hence Definition \ref{def 1} is well-defined. Furthermore, the above argument implies the following lemma.

\begin{lemma}
	A polynomial $f(X)$ is free of binomials if and only if 
	$$\mathrm{ord}^{\mathrm{b}}(f) = \mathrm{ord}(f).$$
	If it is this case, then $X^{m}-1$, where $m = \mathrm{ord}(f)$, is the minimal binomial multiple of $f(X)$.
\end{lemma}

\subsection{The multiple binomial multiple of a squarefree polynomial}
In this subsection we determine the minimal binomial multiple of $f(X)$, in the case that $f(X)$ has no repeated root. Combining Lemma \ref{lem 2}, \ref{lem 1} and Theorem \ref{thm 1}, the minimal binomial multiple and the minimal binomial order of $f(X)$ can be equivalently characterized as follow.

\begin{theorem}\label{thm 4}
	Let $f(X)$ be a squarefree polynomial over $\mathbb{F}_{q}$ with order $m$, and $T_{f} \subseteq \mathbb{Z}/m\mathbb{Z}$ be the primitive defining set of $f(X)$. Let $d$ be the maximal divisor of $m$ such that there exists an equal-difference subset $E \subseteq \mathbb{Z}/m\mathbb{Z}$ over $\mathbb{F}_{q}$ with common difference $d$ which contains $T_{f}$. Then the binomial 
	$$f_{E}(X) = \prod_{\gamma \in E}(X-\zeta_{m}^{\gamma})$$
	induced by $E$ is the minimal binomial multiple of $f(X)$. In particular, $\frac{m}{d}$ is the minimal binomial order of $f(X)$.
\end{theorem}

\begin{proof}
	First we show that the integer $d$ and the set $E$ described as in Theorem \ref{thm 4} exist uniquely. Since $\mathbb{Z}/m\mathbb{Z}$ itself is an equal-difference set containing $T_{f}$, and there are finitely many equal-difference subset of $\mathbb{Z}/m\mathbb{Z}$, then one can choose a maximal divisor $d$ of $m$ such that there exists an equal-difference subset, say
	$$E = \{\gamma,\gamma+d,\cdots,\gamma+(\frac{m}{d}-1)d\} \subseteq \mathbb{Z}/m\mathbb{Z},$$
	that is defined over $\mathbb{F}_{q}$ and contains $T_{f}$.  Assume that
	$$E^{\prime} = \{\gamma^{\prime},\gamma^{\prime}+d,\cdots,\gamma^{\prime}+(\frac{m}{d}-1)d\}$$ 
	is also an equal-difference set containing $T_{f}$. Note that $E$ and $E^{\prime}$ are either equal or disjoint, therefore we have $E = E^{\prime}$.
	
	Let $E$ be defined as above. Following from Lemma \ref{lem 2}, \ref{lem 1} and Theorem \ref{thm 1}, $E$ induces a binomial
	$$f_{E}(X) = (X-\zeta_{m}^{\gamma})(X-\zeta_{m}^{\gamma+d})\cdots(X-\zeta_{m}^{\gamma+(\frac{m}{d}-1)d}) = X^{\frac{m}{d}}-\zeta_{m}^{\frac{\gamma m}{d}}$$
	over $\mathbb{F}_{q}$, which is divided by $f(X)$. Then it only remains to prove that $f_{E}(X)$ is of the lowest degree among the binomials divisible by $f(X)$. Let $X^{k} - \beta$ be a binomial over $\mathbb{F}_{q}$ divided by $f(X)$, with order $M$ and defining set $T \subseteq \mathbb{Z}/M\mathbb{Z}$. As $f(X) \mid X^{k} - \beta$, the order $m$ of $f(X)$ divides $M$, and by Lemma \ref{lem 2}
	$$i_{m,M}(T_{f}) \subseteq T.$$
	And $T_{f} \subseteq E$ implies that $i_{m,M}(T_{f}) \subseteq i_{m,M}(E)$. Without lossing generality, we may write $T$ and $i_{m,M}(E)$ as
	$$T = \{\gamma\cdot\frac{M}{m},\gamma\cdot\frac{M}{m}+D,\cdots,\gamma\cdot\frac{M}{m}+(\frac{M}{D}-1)D\}$$
	and
	$$i_{m,M}(E) = \{\gamma\cdot\frac{M}{m},\gamma\cdot\frac{M}{m}+d\cdot\frac{M}{m},\cdots,\gamma\cdot\frac{M}{m}+d\cdot\frac{M}{m}(\frac{m}{d}-1)\}$$
	respectively, where $\gamma \in T_{f}$ and $D \in M$. Thus we obtain that 
	$$i_{m,M}(T_{f}) \subseteq T\cap i_{m,M}(E) = \{\gamma\cdot\frac{M}{m},\gamma\cdot\frac{M}{m}+D^{\prime},\cdots,\gamma\cdot\frac{M}{m}+(\frac{M}{D^{\prime}}-1)D^{\prime}\},$$
	where $D^{\prime} = \mathrm{lcm}(\frac{M}{m}\cdot d,D)$. Since $\frac{M}{m} \mid D^{\prime}$, then
	$$T_{f} \subseteq \{\gamma,\gamma+\frac{mD^{\prime}}{M},\cdots,\gamma+(\dfrac{m}{mD^{\prime}/M}-1)\cdot \frac{mD^{\prime}}{M}\}.$$
	By the definition of $d$, one has $\frac{mD^{\prime}}{M} \leq d$, which implies that 
	$$k = \frac{M}{D} \geq \frac{M}{D^{\prime}} \geq \frac{m}{d}.$$
\end{proof}

Motivated by Theorem \ref{thm 4}, it is natural to raise the following definition.

\begin{definition}
	With the notations defined as in Theorem \ref{thm 4}, the divisor $d$ of $m$ is called the maximal common difference associated to $T_{f}$, and the set $E$ is called the minimal equal-difference set over $\mathbb{F}_{q}$ containing $T_{f}$.
\end{definition}

From the proof of Theorem \ref{thm 4} we can also recover the following result obtained in \cite{Mart\'{i}nez} (Lemma 3.1, \cite{Mart\'{i}nez}).

\begin{corollary}\label{coro 2}
	Let the notations be given as in Theorem \ref{thm 4}. If a binomial $X^{k}-\beta$ is divided by $f(X)$, then $\frac{m}{d} \mid k$ and $f_{E}(X) \mid X^{k}-\beta$. In particular, $f_{E}(X)$ divides $X^{m}-1$.
\end{corollary}

\begin{proof}
	The last assertion follows from the first one directly. Since $f(X)$ is squarefree, then $f(X) \mid X^{k}-\beta$ implies that $f(X) \mid \mathrm{rad}(X^{k}-\beta)$, which is also a binomial. Without lossing generality, we may assume that $X^{k}-\beta$ is squarefree.
	
	Let $M$ be the order of $X^{k}-\beta$, and 
	$$T = \{\gamma\cdot\frac{M}{m},\gamma\cdot\frac{M}{m}+D,\cdots,\gamma\cdot\frac{M}{m}+(\frac{M}{D}-1)D\} \subseteq \mathbb{Z}/M\mathbb{Z}$$
	be the defining set of $X^{k}-\beta$, where $\gamma \in T_{f}$ and $D \mid M$. Write $D^{\prime} = \mathrm{lcm}(d\cdot\frac{M}{m},D)$. Then by the proof of Theorem \ref{thm 4} one has $D^{\prime} \leq d\cdot\frac{M}{m}$, which implies that $D^{\prime} = d\cdot\frac{M}{m}$ and $D \mid d\cdot\frac{M}{m}$. As $k = \frac{M}{D}$, then $\frac{m}{d} \mid k$. Moreover, by the definition of $E$ one can check directly that $i_{m,M}(E) \subseteq T$, which yields that $f_{E}(X)$ divides $X^{k}-\beta$. 
\end{proof}

Now we compute explicitly the minimal binomial multiple of a squarefree polynomial $f(X)$. We begin with the case where $f(X)$ is irreducible.

\begin{lemma}\label{lem 5}
	Let $f(X)$ be an irreducible polynomial over $\mathbb{F}_{q}$ with order $m$, and $c_{m/q}(\gamma)$ be the $q$-cyclotomic coset modulo $m$ which defines $f(X)$. Then the minimal binomial multiple of $f(X)$ is given by
	$$X^{\frac{m}{d_{f}}}-\zeta_{m}^{\frac{\gamma m}{d_{f}}},$$
	where $d_{f} = \mathrm{gcd}(m,q-1)$. In particular, the minimal binomial order of $f(X)$ is $\frac{m}{d_{f}}$.
\end{lemma}

\begin{proof}
	Notice that the defining set of $X^{\frac{m}{d_{f}}}-\zeta_{m}^{\frac{\gamma m}{d_{f}}}$ is 
	$$E_{f} = \{\gamma,\gamma+d_{f},\cdots,\gamma+(\frac{m}{d_{f}}-1)d_{f}\},$$
	therefore by Theorem \ref{thm 4} it suffices to show that $E_{f}$ is the minimal equal-difference set over $\mathbb{F}_{q}$ containing $T_{f}$.
	
	Since the order of $f(X)$ is $m$, then $\gamma$ is coprime to $m$. Any equal-difference subset $E^{\prime} \subseteq \mathbb{Z}/m\mathbb{Z}$ over $\mathbb{F}_{q}$ contains $T_{f}$ if and only if $\gamma \in E^{\prime}$ and the common difference $d^{\prime}$ of $E^{\prime}$ divides $q-1$. As $d^{\prime} \mid m$, the largest possible value of $d^{\prime}$ is equal to $\mathrm{gcd}(m,q-1)=d_{f}$. On the other hand, by Lemma \ref{lem 2} the set $E_{f}$ is indeed an equal-difference subset over $\mathbb{F}_{q}$ containing $T_{f}$. Hence it is the minimal equal-difference set over $\mathbb{F}_{q}$ containing $T_{f}$.
\end{proof}

\begin{corollary}
	If two irreducible polynomials have the same order, then they also have the same minimal binomial order.
\end{corollary}

Next we consider the minimal binomial multiple of the least common multiple of two squarefree polynomials, with the minimal binomial multiples of these two polynomials being known.

\begin{proposition}\label{prop 1}
	Let $f_{1}(X), f_{2}(X) \in \mathbb{F}_{q}[X]$ be squarefree. For $i=1,2$, let $m_{i}$, $n_{i}$ and $T_{i}$ denote the order, the minimal binomial order and the primitive defining set of $f_{i}(X)$ respectively. Choose a $\gamma \in T_{1}$ and a $\eta \in T_{2}$. Set $m =\mathrm{lcm}(m_{1},m_{2})$ and 
	$$d = \mathrm{gcd}(\frac{m\eta}{m_{2}}-\frac{m\gamma}{m_{1}},\frac{m}{n_{1}},\frac{m}{n_{2}}).$$
	Then the minimal binomial order of $f(X)=\mathrm{lcm}(f_{1}(X),f_{2}(X))$ is $\frac{m}{d}$, and the minimal binomial multiple of $f(X)$ is
	$$X^{\frac{m}{d}} - \zeta_{m}^{\frac{\gamma m^{2}}{m_{1}d}}.$$
\end{proposition}

\begin{proof}
	It is obvious to see that the order of $f(X)$ is $m$, and the primitive defining set of $f(X)$ is
	$$T_{f} = i_{m_{1},m}(T_{1}) \cup i_{m_{2},m}(T_{2}).$$
	Notice that the defining set of $X^{\frac{m}{d}} - \zeta_{m}^{\frac{\gamma m^{2}}{m_{1}d}}$ is
	$$E = \{\gamma\cdot\frac{m}{m_{1}},\gamma\cdot\frac{m}{m_{1}}+d,\cdots,\gamma\cdot\frac{m}{m_{1}}+(\frac{m}{d}-1)d\} \subseteq \mathbb{Z}/m\mathbb{Z}.$$
	Then it suffices to show that $E$ is the minimal equal-difference set over $\mathbb{F}_{q}$ containing $T_{f}$.
	
	By the proof of Theorem \ref{thm 4} the minimal equal-defining sets over $\mathbb{F}_{q}$ containing $T_{1}$ and $T_{2}$ are respectively given by
	$$E_{1} = \{\gamma,\gamma+\frac{m_{1}}{n_{1}},\cdots,\gamma+(n_{1}-1)\frac{m_{1}}{n_{1}}\}$$
	and 
	$$E_{2} = \{\eta,\eta+\frac{m_{2}}{n_{2}},\cdots,\eta+(n_{2}-1)\frac{m_{2}}{n_{2}}\}.$$
	As $\gamma\cdot\frac{m}{m_{1}} \in E$ and 
	$$d \mid \frac{m}{n_{1}} = \frac{m_{1}}{n_{1}}\cdot\frac{m}{m_{1}},$$
	then 
	$$i_{m_{1},m}(T_{1}) \subseteq i_{m_{1},m}(E_{1}) \subseteq E.$$
	As $d \mid \frac{m\eta}{m_{2}} - \frac{m\gamma}{m_{1}}$, $E$ contains $\frac{m\eta}{m_{2}}$. By the same argument we obtain that 
	$$i_{m_{2},m}(T_{2}) \subseteq i_{m_{2},m}(E_{2}) \subseteq E.$$
	Moreover, since $E_{1}$ is defined over $\mathbb{F}_{q}$, $\frac{m_{1}}{n_{1}} \mid \gamma(q-1)$, which indicates that
	$$\frac{m}{n_{1}} = \frac{m_{1}}{n_{1}}\cdot\frac{m}{m_{1}} \mid \gamma\cdot\frac{m}{m_{1}}(q-1).$$
	Hence $E$ is an equal-difference set over $\mathbb{F}_{q}$ containing $T_{f}$.
	
	On the other hand, if $E^{\prime} \subseteq \mathbb{Z}/m\mathbb{Z}$ is an equal-difference set over $\mathbb{F}_{q}$ containing $T_{f}$, then by Corollary \ref{coro 2} we have the inclusions
	$$i_{m_{1},m}(E_{1}) \subseteq E, \ \mathrm{and} \ i_{m_{2},m}(E_{2}) \subseteq E.$$
	It follows that the common difference $d^{\prime}$ of $E^{\prime}$ divides both $\frac{m}{n_{1}}$, $\frac{m}{n_{2}}$ and $\frac{m\eta}{m_{2}}-\frac{m\gamma}{m_{1}}$, or equivalently, $d^{\prime}$ divides $d = \mathrm{gcd}(\frac{m\eta}{m_{2}}-\frac{m\gamma}{m_{1}},\frac{m}{n_{1}},\frac{m}{n_{2}})$. Consequently, $E$ is the minimal equal-difference set over $\mathbb{F}_{q}$ containing $T_{f}$.
\end{proof}

From Proposition \ref{prop 1} we deduce that the property of freeness of binomials is closed under taking the least common multiple, which generalized Lemma $3.4$ in \cite{Mart\'{i}nez}.

\begin{corollary}
	If both $f_{1}(X)$ and $f_{2}(X)$ are free of binomials then so is $\mathrm{lcm}(f_{1}(X),f_{2}(X))$.
\end{corollary}

\begin{proof}
	We adopt the notations from Proposition \ref{prop 1}. Since $f_{1}(X)$ and $f_{2}(X)$ are free of binomials, then $n_{1}=m_{1}$ and $n_{2} = m_{2}$. Therefore one has
	$$\mathrm{gcd}(\frac{m}{n_{1}},\frac{m}{n_{2}}) = \mathrm{gcd}(\frac{\mathrm{lcm}(m_{1},m_{2})}{m_{1}},\frac{\mathrm{lcm}(m_{1},m_{2})}{m_{2}}) = \mathrm{gcd}(\frac{m_{2}}{\mathrm{gcd}(m_{1},m_{2})},\frac{m_{1}}{\mathrm{gcd}(m_{1},m_{2})})=1,$$
	which implies that $d=1$. It then follows that $m \mid \frac{\gamma m^{2}}{m_{1}}$ and $\zeta_{m}^{\frac{\gamma m^{2}}{m_{1}}} = 1$.
\end{proof}

Combining Lemma \ref{lem 5} and \ref{prop 1} we obtain the minimal binomial order and the minimal binomial multiple of any squarefree polynomial over $\mathbb{F}_{q}$, which are given by the theorem below.

\begin{theorem}\label{thm 5}
	Let $f(X)$ be a squarefree polynomial over $\mathbb{F}_{q}$. Assume that the order of $f(X)$ is $m$, and the defining set $T_{f} \subseteq \mathbb{Z}/m\mathbb{Z}$ of $f(X)$ is divided into $q$-cyclotomic cosets modulo $m$ as
	$$T_{f} = \bigsqcup_{i=1}^{t}c_{m/q}(\gamma_{i}).$$
	Define $d_{i}=\gamma_{i+1}-\gamma_{i}$ for $i=1,\cdots,t-1$, and
	$$d_{f} = \mathrm{gcd}(d_{1},\cdots,d_{t-1},m,q-1).$$
	Then the minimal binomial order of $f(X)$ is $\frac{m}{d_{f}}$, and the minimal binomial multiple of $f(X)$ is given by
	$$X^{\frac{m}{d_{f}}} - \zeta_{m}^{\frac{\gamma m}{d_{f}}}.$$
\end{theorem}

\begin{proof}
	Let $f_{i}(X)$ be the irreducible polynomial defined by the $q$-cyclotomic coset $c_{m/q}(\gamma_{i})$ modulo $m$. Then the primitive defining set of $f_{i}(X)$ is $c_{m_{i}/q}(\gamma_{i}^{\prime})$, where $m_{i} = \frac{m}{\mathrm{gcd}(m,\gamma_{i})}$ and $\gamma_{i}^{\prime} = \frac{\gamma_{i}}{\mathrm{gcd}(m,\gamma_{i})}$, and the minimal equal-difference set over $\mathbb{F}_{q}$ containing $c_{m_{i}/q}(\gamma_{i}^{\prime})$ is
    $$E_{f_{i}}=\{\gamma_{i}^{\prime},\gamma_{i}^{\prime}+d_{f_{i}},\cdots,\gamma_{i}^{\prime}+(\frac{m_{i}}{d_{f_{i}}}-1)d_{f_{i}}\},$$
    where $d_{f_{i}} = \mathrm{gcd}(m_{i},q-1)$. Notice that the order of $f(X)$ is $m$, therefore one has
    $$\mathrm{lcm}(m_{1},\cdots,m_{t}) = m.$$
    It follows that the maximal common difference associated to $T_{f}$ is
    \begin{align*}
    	d_{f} &= \mathrm{gcd}(\frac{m\gamma_{1}^{\prime}}{m_{1}}-\frac{m\gamma_{2}^{\prime}}{m_{2}},\cdots,\frac{m\gamma_{t-1}^{\prime}}{m_{t-1}}-\frac{m\gamma_{t}^{\prime}}{m_{t}},\frac{m d_{f_{1}}}{m_{1}},\cdots,\frac{m d_{f_{t}}}{m_{t}})\\
    	&=\mathrm{gcd}(d_{1},\cdots,d_{t},\mathrm{gcd}(m,\frac{m}{m_{1}}(q-1)),\cdots,\mathrm{gcd}(m,\frac{m}{m_{t}}(q-1)))\\
    	&=\mathrm{gcd}(d_{1},\cdots,d_{t},m,q-1).
    \end{align*}
    Here the last equality also follows from the fact that $\mathrm{lcm}(m_{1},\cdots,m_{t}) = m$.
\end{proof}

\begin{corollary}\label{coro 1}
	The following statements are equivalent.
	\begin{description}
		\item[(1)] The polynomial $f(X)$ is free of binomial.
		\item[(2)] $d_{f} = \mathrm{gcd}(d_{1},\cdots,d_{t-1},m,q-1)=1$.
		\item[(3)] The minimal equal-difference set containing $T_{f}$ is $\mathbb{Z}/m\mathbb{Z}$.
	\end{description}
\end{corollary}

\begin{proof}
	The polynomial $f(X)$ is free of binomial if and only if
	$$m = \mathrm{ord}(f) = \mathrm{ord}^{\mathrm{b}}(f) = \frac{m}{d_{f}},$$
	that is, $d_{f}=1$. On the other hand, $d_{f}=1$ is equivalent to that the minimal equal-difference set containing $T_{f}$ is
	$$E_{f} = \{\gamma,\gamma+1,\cdots,\gamma+m-1\} = \mathbb{Z}/m\mathbb{Z}.$$
\end{proof}

\subsection{The general case}
Let $f(X)$ be given by
$$f(X) = f_{1}(X)^{h_{1}}\cdots f_{t}(X)^{h_{t}},$$
where $f_{1}(X),\cdots,f_{t}(X)$ are pairwise distinct irreducible polynomials over $\mathbb{F}_{q}$ and $h_{1},\cdots,h_{t}$ are positive integers. Define the radical of $f(X)$ to be 
$$\mathrm{rad}(f) = f_{1}(X)\cdots f_{t}(X).$$

\begin{lemma}\label{lem 3}
	Assume that $\mathrm{rad}(f)$ has minimal binomial order $\mathrm{ord}^{\mathrm{b}}(\mathrm{rad}(f)) = n$, and minimal binomial multiple $X^{n}-\lambda$. Set $u$ to be the smallest positive integer such that $p^{u} \geq h_{i}$ for all $i=1,\cdots,r$. Then the minimal binomial order of $f(X)$ is $\mathrm{ord}^{\mathrm{b}}(f) = p^{u}n$, and the minimal binomial multiple of $f(X)$ is $X^{p^{u}n}-\lambda^{p^{u}}$.
\end{lemma}

\begin{proof}
	As $\mathrm{rad}(f) \mid X^{n}-\lambda$ and $h_{i} \leq p^{u}$ for $i=1,\cdots,r$, then every root of $f(X)$ is a root of 
	$$(X^{n}-\lambda)^{p^{u}} = X^{p^{u}n}-\lambda^{p^{u}},$$
	with the multiplicity in $f(X)$ not larger than that with respect to $X^{p^{u}n}-\lambda^{p^{u}}$. Consequently $f(X)$ divides $X^{p^{u}n}-\lambda^{p^{u}}$. By the definition of minimal binomial order we have
    \begin{equation}\label{eq 3}
	    \mathrm{ord}^{\mathrm{b}}(f) \leq p^{u}n.
    \end{equation}

	To prove the opposite direction, suppose that $X^{N}-\beta$ is the minimal binomial multiple of $f(X)$, where $N = \mathrm{ord}^{\mathrm{b}}(f)$. Write $N = p^{v}N^{\prime}$ where $v \geq 0$ and $N^{\prime}$ is coprime to $p$. Then there is a nonzero element $\beta^{\prime} \in \mathbb{F}_{q}$ such that 
	$$X^{N}-\beta = (X^{N^{\prime}}-\beta^{\prime})^{p^{v}}.$$
	Notice that $f(X) \mid (X^{N^{\prime}}-\beta^{\prime})^{p^{v}}$ if and only if $\mathrm{rad}(f) \mid X^{N^{\prime}}-\beta^{\prime}$ and $h_{i} \leq p^{v}$ for $i=1,\cdots,t$. Therefore from the definition of $n$ and $u$ we obtain that $n \leq N^{\prime}$ and $u \leq v$, which indicates that
	\begin{equation}\label{eq 4}
		\mathrm{ord}^{\mathrm{b}}(f) = p^{v}N^{\prime} \geq p^{u}n.
	\end{equation}
	Then \eqref{eq 3} and \eqref{eq 4} together gives that $\mathrm{ord}^{\mathrm{b}}(f) = p^{u}n$.
	
	Finally, as $\mathrm{ord}^{\mathrm{b}}(f) = p^{u}n$ and $f(X) \mid X^{p^{u}n}-\lambda^{p^{u}}$, the uniqueness of the minimal binomial multiple guarantees that $X^{p^{u}n}-\lambda^{p^{u}}$ is exactly the minimal binomial multiple of $f(X)$.
\end{proof}

Theorem \ref{thm 5} and Lemma \ref{lem 3} together implies the formula for the minimal binomial order and the minimal binomial multiple of $f(X)$ in the general case.

\begin{theorem}\label{thm 7}
	Let $f(X)$ be a nonconstant polynomial over $\mathbb{F}_{q}$, which is coprime to $X$, with the irreducible factorization over $\mathbb{F}_{q}$ given by
	$$f(X) = f_{1}(X)^{h_{1}}\cdots f_{t}(X)^{h_{t}}.$$
	Let $m$ be the order of $\mathrm{rad}(f)$, and $c_{m/q}(\gamma_{i})$ be the $q$-cyclotomic coset modulo $m$ defining the irreducible factor $f_{i}(X)$ for $1 \leq i \leq t$. Set $u$ to be the smallest integer such that $p^{u} \geq h_{i}$ for all $i = 1,\cdots,t$, $d_{i} = \gamma_{i+1}-\gamma_{i}$ for $i=1,\cdots,t-1$, and
	$$d_{f} = \mathrm{gcd}(d_{1},\cdots,d_{t-1},m,q-1).$$
	Then the minimal binomial order of $f(X)$ is $p^{u}\cdot\frac{m}{d_{f}}$, and the minimal binomial multiple of $f(X)$ is given by
	$$X^{\frac{p^{u}m}{d_{f}}}-\zeta_{m}^{\frac{p^{u}\gamma m}{d_{f}}},$$
	where $\gamma$ can be choose to be any $\gamma_{i}$, $1 \leq i \leq t$.
\end{theorem}

Corollary \ref{coro 2} and \ref{coro 1} can be generalized in the general case.

\begin{corollary}
	Let the notations be given as in Theorem \ref{thm 7}. If a binomial $X^{k}-\beta$ is divided by $f(X)$, then $\frac{m}{d} \mid p^{u}\cdot\frac{m}{d_{f}}$ and $$X^{\frac{p^{u}m}{d_{f}}}-\zeta_{m}^{\frac{p^{u}\gamma m}{d_{f}}} \mid X^{k}-\beta.$$ 
	In particular, 
	$$X^{\frac{p^{u}m}{d_{f}}}-\zeta_{m}^{\frac{p^{u}\gamma m}{d_{f}}} \mid X^{m}-1.$$
\end{corollary}

\begin{corollary}
	The following statements are equivalent.
	\begin{description}
		\item[(1)] The polynomial $f(X)$ is free of binomial.
		\item[(2)] The radical $\mathrm{rad}(f)$ of $f(X)$ is free of binomials.
		\item[(3)] $d_{f} = \mathrm{gcd}(d_{1},\cdots,d_{t-1},m,q-1)=1$.
		\item[(4)] The minimal equal-difference set containing $T_{f}$ is $\mathbb{Z}/m\mathbb{Z}$.
	\end{description}
\end{corollary}

\section{Constacyclic codes of minimal distance $2$ over $\mathbb{F}_{q}$}\label{sec 1}
Let $N$ be a positive integer, and $\lambda \in \mathbb{F}_{q}^{\ast}$. In this section we apply the results obtained in the last two sections to determine all the $\lambda$-constacyclic codes of length $N$ over $\mathbb{F}_{q}$ which has minimal distance $2$. To deal with both simple-rooted and repeat-rooted constacyclic codes, we allow $N$ to be divisible by $p$.

Recall that a $\lambda$-constacyclic code $C$ of length $N$ can be identified with an ideal 
$$(f(X)) = \{f(X)g(X) \ | \ \mathrm{deg}(g(X)) < N-\mathrm{deg}(f(X))\} \subseteq \mathbb{F}_{q}[X]/(X^{N}-\lambda),$$
where $f(X)$ is a divisor of $X^{N}-\lambda$. The code $C = (f(X))$ has minimal distance $2$ if and only if it has an element of the form $X^{k}-\alpha X^{k^{\prime}}$, where $k^{\prime} < k <N$ and $\lambda \neq 0$. Note that $f(X)$ is coprime to $X$, therefore $f(X) \mid X^{k}-\alpha X^{k^{\prime}}$ amounts to $f(X) \mid X^{k-k^{\prime}}-\alpha$, i.e, 
$$X^{k-k^{\prime}}-\alpha \in (f(X)).$$
Together with Corollary \ref{coro 2}, we prove the following criterion for $f(X)$ to be of minimal distance $2$.

\begin{proposition}\label{prop 2}
	The $\lambda$-constacyclic code $(f(X))$ of length $N$ has minimal distance $2$ if and only if the minimal binomial multiple of $f(X)$ is a proper factor of $X^{N}-\lambda$.
\end{proposition}

Now we compute the proper factors in the form of a binomial of $X^{N}-\lambda$. Write $N$ as $N = p^{v}N^{\prime}$, where $v = v_{p}(N) \geq 0$ and $n$ is coprime to $q$. Let $r$ be the order of $\lambda$, and let $t$ be the unique integer satisfying that $tp^{v} \equiv 1 \pmod{r}$ and $1 \leq t \leq r-1$. Then $X^{N}-\lambda$ can be expressed as
$$X^{N}-\lambda = (X^{n}-\lambda^{\prime})^{p^{v}},$$
where $\lambda^{\prime} = \lambda^{t}$. As stated in Section \ref{Sec 2}, there is a primitive $nr$-th root $\zeta_{nr}^{\delta}$ of unity such that $\zeta_{nr}^{\delta n}=\lambda^{\prime}$, where $\delta$ is a positive integer coprime to $nr$, and the defining set of $X^{n^{\prime}}-\lambda^{\prime}$ modulo $nr$ is
$$i_{m,nr}(T) = \{\delta,\delta+r,\cdots,\delta+r(n-1)\} = \delta + r\cdot\mathbb{Z}/nr\mathbb{Z}.$$
Here $m$ denotes the order of $X^{n}-\lambda^{\prime}$. It is trivial to see that any equal-difference subset of $i_{m,nr}(T)$ can be written as 
$$E = \{\gamma+rj,\gamma+rj+rd,\cdots,\gamma+rj+(\frac{n}{d}-1)rd\},$$
where $d$ is a positive divisor of $n$ and $0 \leq j \leq d-1$. By Lemma \ref{lem 1} the set $E$ is defined over $\mathbb{F}_{q}$ if and only if 
\begin{equation}\label{eq 5}
	(\gamma+rj)q \equiv \gamma+rj \pmod{rd}.
\end{equation}
As $r \mid q-1$, \eqref{eq 5} is equivalent to
\begin{equation}\label{eq 6}
	d \mid (\gamma+rj)\cdot\frac{q-1}{r}.
\end{equation}
Write $g_{d} = \mathrm{gcd}(d,\frac{q-1}{r})$. Since $\delta$ is coprime to $nr$, then $\delta+rj$ is coprime to $r$. It follows that \eqref{eq 6} holds if and only if $\mathrm{gcd}(\frac{d}{g_{d}},r)=1$ and 
$$rj \equiv -\delta \pmod{\frac{d}{g_{d}}}.$$

\begin{lemma}\label{lem 7}
	Let the notations be given as above. Then a proper factor, which is a binomial, of $X^{N}-\lambda$ is in either of the following form:
	\begin{description}
		\item[(1)] $(X^{n}-\lambda^{\prime})^{p^{u}}$, where $0 \leq u < v$;
		\item[(2)] $(X^{\frac{n}{d}}-\zeta_{nr}^{(\delta+rj)\frac{n}{d}})^{p^{w}}$, where $d > 1$ is a divisor of $n$ such that $\mathrm{gcd}(\frac{d}{g_{d}},r)=1$, $j$ is an integer such that $rj \equiv -\delta \pmod{\frac{d}{g_{d}}}$ and $0 \leq j \leq d-1$, and $0 \leq w \leq v$.
	\end{description}
\end{lemma}

Now combining Theorem \ref{thm 2}, Proposition \ref{prop 2} and Lemma \ref{lem 7} gives the $\lambda$-constacyclic codes of length $N$ with minimal distance $2$. We write $n^{\prime} = \mathrm{gcd}(n,r)$ and 
$$n = n^{\prime}\cdotp_{1}^{e_{1}}\cdots p_{s}^{e_{s}},$$ 
and for any divisor $d > 1$ of $d$ such that $\mathrm{gcd}(\frac{d}{g_{d}},r)=1$ write $n_{d}^{\prime} = \mathrm{gcd}(\frac{n}{d},r)$ and 
$$\frac{n}{d} = n_{d}^{\prime}\cdot p_{d,1}^{e_{d,1}}\cdots p_{d,s_{d}}^{e_{d,s_{d}}}.$$
Here $p_{1},\cdots,p_{s}$ (resp. $p_{d,1},\cdots,p_{s,s_{d}}$) are pairwise distinct prime numbers not dividing $q$ and $r$, and $e_{1},\cdots,e_{s}$ (resp. $e_{d,1},\cdots,e_{s,s_{d}}$) are positive integers. We adopt the notations in Section \ref{Sec 2}.
 
\begin{theorem}\label{thm 6}
	With the notations as above, a $\lambda$-constacyclic code of length $N$ with minimal distance $2$ is either of the following form:
	\begin{description}
		\item[(1)] 
		$$\left(\prod_{\mathbf{x} \in \Gamma(n,r)}\prod_{\mathbf{m} \in [0,e_{1}]\times\cdots\times[0,e_{s}]}\prod_{R_{\mathbf{m},\mathbf{i},\mathbf{T}}(\delta_{\mathbf{x}}) \in \mathcal{CR}_{nr/q}^{\mathbf{m}}(\delta_{\mathbf{x}})}M_{R_{\mathbf{m},\mathbf{i},\mathbf{T}}(\delta_{\mathbf{x}})}(X)^{\Omega_{R_{\mathbf{m},\mathbf{i},\mathbf{T}}(\delta_{\mathbf{x}})}}\right),$$
		where $0 \leq \Omega_{R_{\mathbf{m},\mathbf{i},\mathbf{T}}(\delta_{\mathbf{x}})} \leq p^{v-1}$;
		\item[(2)] 
		$$\left(\prod_{\mathbf{x} \in \Gamma(\frac{n}{d},r)}\prod_{\mathbf{m} \in [0,e_{d,1}]\times\cdots\times[0,e_{d,s_{d}}]}\prod_{R_{\mathbf{m},\mathbf{i},\mathbf{T}}((\delta+rj)_{\mathbf{x}}) \in \mathcal{CR}_{\frac{nr}{d}/q}^{\mathbf{m}}((\delta+rj)_{\mathbf{x}})}M_{R_{\mathbf{m},\mathbf{i},\mathbf{T}}((\delta+rj)_{\mathbf{x}})}(X)^{\Omega_{R_{\mathbf{m},\mathbf{i},\mathbf{T}}((\delta+rj)_{\mathbf{x}})}}\right),$$
		where $d > 1$ is a divisor of $n$ such that $\mathrm{gcd}(\frac{d}{g_{d}},r)=1$, $j$ is an integer such that $rj \equiv -\delta \pmod{\frac{d}{g_{d}}}$ and $0 \leq j \leq d-1$, and $0 \leq \Omega_{R_{\mathbf{m},\mathbf{i},\mathbf{T}}((\delta+rj)_{\mathbf{x}})} \leq p^{v-1}$.
	\end{description}
	Here the polynomial $M_{R_{\mathbf{m},\mathbf{i},\mathbf{T}}(\delta_{\mathbf{x}})}(X)$ and $M_{R_{\mathbf{m},\mathbf{i},\mathbf{T}}((\delta+rj)_{\mathbf{x}})}(X)$ are given explicitly by Theorem \ref{thm 2}.
\end{theorem}

\begin{remark}
	Theorem \ref{thm 6} produces all the $\lambda$-constacyclic codes of length $N$ with minimal distance $2$, however, the codes given by $(1)$ and $(2)$ may repeat.
\end{remark}

\section*{Acknowledgment}
This work was supported by Natural Science Foundation of Beijing Municipal(M23017). 

\section*{Data availability}
Data sharing not applicable to this article as no datasets were generated or analysed during the current study.

\section*{Declaration of competing interest}
The authors declare that we have no known competing financial interests or personal relationships that 
could have appeared to influence the work reported in this paper.

\end{document}